\documentclass{article}

\usepackage[english]{babel}

\usepackage{fullpage}
\usepackage{geometry}
\usepackage{xcolor}
\usepackage{graphicx}
\usepackage{multirow}
\usepackage{wrapfig}
\usepackage{longtable}
\usepackage{float}
\usepackage{setspace} 
\usepackage{array}
\usepackage{booktabs,tabularx}
\usepackage{amsmath,amssymb,enumerate,url,amsthm}
\usepackage{ytableau}
\usepackage[numbers,sort]{natbib}
\usepackage{subcaption}
\usepackage{gensymb}
\usepackage{enumitem}
\usepackage{pgf,tikz, varwidth}
\usetikzlibrary{arrows, matrix}
\usetikzlibrary{shapes.symbols, calc, shapes, decorations}
\usepackage{chngcntr}
\usepackage{hyperref,cases}

\usepackage[bottom]{footmisc}

\definecolor{dblue}{RGB}{0,54,135}
\definecolor{lorange}{RGB}{255,155,0}

\newcommand{\cp}{\mathrm{cp}}
\newcommand{\CP}{\mathcal{CP}}

\newcommand{\Z}{\mathbb{Z}}

\newcommand{\dcom}[1]{{\color{green}{DE: #1}} }

\theoremstyle{definition}
\newtheorem*{remark}{Remark}
\newtheorem{theorem}{Theorem}
\newtheorem{cor}[theorem]{Corollary}
\newtheorem{lemma}[theorem]{Lemma}

\newtheorem*{defn}{Definition}
\newtheorem{example}[theorem]{Example}
\newtheorem{conj}[theorem]{Conjecture}
\numberwithin{theorem}{section}
\newtheorem{quest}[theorem]{Question}

\counterwithin{table}{section}

\allowdisplaybreaks

\title{On the parity of the number of $(a,b,m)$-copartitions of $n$}
\author{Hannah E. Burson and Dennis Eichhorn}
\date{\today}

\begin{document}
\maketitle

\begin{abstract}
    We continue the study of the $(a,b,m)$-copartition function $\cp_{a,b,m}(n)$, 
    which arose as a combinatorial generalization of Andrews' partitions with even parts below odd parts.
 The generating function of $\cp_{a,b,m}(n)$ has a nice representation as an infinite product. 
In this paper, we focus on the parity of $\cp_{a,b,m}(n)$. 
As with the ordinary partition function, it is difficult to show positive density of either even or odd values of $\cp_{a,b,m}(n)$
for arbitrary $a, b$, and $m$. 
However, we find specific cases of $a,b,m$ such that $\cp_{a,b,m}(n)$ is even with density 1. Additionally, we show that the sequence $\{\cp_{a,m-a,m}(n)\}_{n=0}^\infty$ takes both even and odd values infinitely often.
\end{abstract}
\section{Introduction}
In \cite{BursonEichhorn}, the authors introduce and develop the theory of copartitions, which have connections to mock theta functions, Rogers-Ramanujan partitions, the capsids of Garvan and Schlosser \cite{Capsids}, and many other classical partition-theoretic objects.
Each $(a,b,m)$-copartition is comprised of three partitions: 
a partition into parts $\equiv  a \pmod m$,
a partition into parts $\equiv b \pmod m$, and a rectangular partition that unites them. Additionally, the $(a,b,m)$-copartition generating function can be written as an infinite product, further motivating interest in the copartition counting functions.
As with any type of partitions, one of the first questions one may ask is about how frequently these counting functions take even and odd values. 

For example, Kolberg showed that $p(n)$, the ordinary partition function, takes both even and odd values infinitely often \cite{Kolberg}.
Parkin and Shanks studied the parity of $p(n)$ computationally, and their evidence strongly suggests that $p(n)$ is even about half of the time \cite{ParkinShanks}.
Sadly, the best known results on the parity of $p(n)$ are spectacularly far from proving anything of the sort \cite{BellaicheGreenSound}. In fact,
it is still an open problem to even show that $p(n)$ is even or odd with positive density.

In this paper, we study the parity of the copartition counting functions.
Although we conjecture that some copartition functions are equally often even and odd, we show that this is not the case for all such functions. 
Furthermore, in some special cases, we are able to demonstrate explicit sets with positive density on which $\cp_{a,b,m}(n)$, the number of $(a,b,m)$-copartitions of $n$, is always even.
These explicit sets allow us to give infinitely many arithmetic progressions on which certain copartition functions are always even.
For example, we show that for $r=3, 17, 24, 31, 38, 45$, we have 
$$ \cp_{3,1,4}(49k+r) \equiv 0 \pmod 2,
$$
and
for $s=9, 14, 19, 24$, we have
$$ \cp_{5,1,6}(25k+s) \equiv 0 \pmod 2
$$
for every nonnegative integer $k$.

In Section 2, we recall the definition and generating function for copartitions.
In Section 3, we recall conjugation, discuss self-conjugate copartitions, and give lower bounds on the number of even values of $\cp_{a,a,m}(n)$.
In Section 4, we focus on $\cp_{a,m-a,m}(n)$; we show that $\cp_{a,m-a,m}(n)$ takes both odd and even values infinitely often,
give explicit sets with density one on which $\cp_{3,1,4}(n)$ and $\cp_{5,1,6}(n)$ are even, and provide infinitely many congruences modulo two in arithmetic progressions for each of these functions.
Sections 3 and 4 also include several open problems.
In Appendix \ref{sec:datatables},
we provide computational data surrounding our open questions and conjectures.

\section{Background on Copartitions}

In this section, we review $(a,b,m)$-copartitions and the function $\cp_{a,b,m}(n),$ which were first introduced in \cite{BursonEichhorn}. 
\begin{defn} 
An $(a,b,m)$-copartition is a triple of partitions $(\gamma,\rho,\sigma)$, where each of the parts of $\gamma$ is at least $a$ and congruent to $a\pmod m$, each of the parts of $\sigma$ is at least $b$ and congruent to $b\pmod m$, and $\rho$ has the same number of parts as $\sigma$, each of which have size equal to $m$ times the number of parts of $\gamma$.\\
 When $a,b,m\ge 1$, we let $\cp_{a,b,m}(n)$ denote the number of $(a,b,m)$-copartitions of size $n$. 
\end{defn}

Although $\rho$ is completely determined by $\gamma$ and $\sigma$, the graphical representation we use suggests that the natural way to write the triple is $(\gamma,\rho,\sigma)$.  
We refer to $\gamma$ as the \emph{ground} of a copartition and $\sigma$ as the \emph{sky}.

\begin{example}

The $(2,1,3)$-copartitions of size 9 are
   $$ \left (\{5,2^2\},\emptyset,\emptyset \right ),\;
   \left (\{5\},\{3\},\{1\} \right ),\;
     \left (\{2\},\{3\},\{4\} \right ),\;
    \left  (\emptyset,\emptyset, \{7+1^{2}\} \right ),\;
   $$ 
   $$
   \left (\emptyset,\emptyset, \{4^2+1\} \right ),
   \left (\emptyset,\emptyset, \{4+1^5\} \right ),
   \text{ and }  \left (\emptyset,\emptyset, \{1^9\} \right ).
$$
Thus, $\cp_{2,1,3}(9)=7$. 

\end{example}
  
To represent the $(a,b,m)$ copartition $(\sigma, \rho, \gamma)$ graphically, we append the $m$-modular diagram for $\sigma$ to the right of the $m$-modular diagram for $\rho$. Then, we append the conjugate of the $m$-modular diagram for $\gamma$ below $\rho$. 

\begin{example} The following diagram represents the $(a,b,m)$-copartition $(\{3m+a,2m+a,2m+a,a\},\{4m,4m\},\{3m+b,2m+b\})$.

\begin{center}
\begin{minipage}{0.35\textwidth}
\begin{tikzpicture}[remember picture]

\node (n) {\begin{varwidth}{5cm}{
\begin{ytableau}
m&m&m&m&b&m&m&m\\
m&m&m&m&b&m&m\\
a&a&a&a\\
m&m&m\\
m&m&m\\
m
\end{ytableau}} \end{varwidth}};
\draw[very thick, black] ([xshift=0.4em, yshift=+1.5em] n.west)--([xshift=6.5em,yshift=1.5em]n.west)--([xshift=0em,yshift=-.4em] n.north);

\end{tikzpicture}
\end{minipage}
\end{center}
\end{example}

In \cite{BursonEichhorn}, we also show that
\begin{align}\label{eq:cpgf}
   {\mathbf{cp}}_{a,b,m}(q)&:= \sum_{n=0}^\infty \cp_{a,b,m}(n)q^n
   =\frac{(q^{a+b};q^m)_\infty}{(q^b;q^m)_\infty(q^a;q^m)_\infty}.
\end{align}

\section{Conjugation, Self-Conjugate Copartitions, and the Parity of $\cp_{a,a,m}(n)$}

In this section, we explore the parity of $\cp_{a,a,m}(n)$. Many of our results follow from considering the fixed points of the conjugation involution on $\cp_{a,a,m}(n)$. 
As in \cite{BursonEichhorn}, we define the conjugate of a copartition $(\gamma,\rho,\sigma)$ as the copartition obtained by reflecting the graphical representation about the line $y=-x$. 
Defining $\nu(\lambda)$ to be the number of parts of a partition $\lambda$, 
the conjugate copartition is precisely $(\sigma,\rho',\gamma),$ where $\rho'$ consists of exactly $\nu(\gamma)$ parts of size $m \times \nu(\sigma)=m \times \nu(\rho)$. Equivalently, $\rho'$ is the partition obtained by conjugating the $m$-modular diagram representing $\rho$.

\begin{figure}[htb]
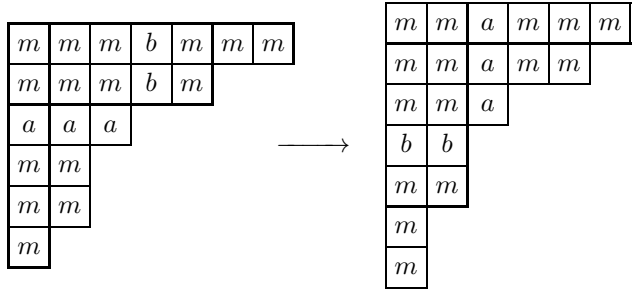

\centering
\begin{minipage}{10em}
      \begin{ytableau}
        m&m&m&b&m&m&m\\
        m&m&m&b&m\\
        a&a&a\\
        m&m\\
        m&m\\
        m
        \end{ytableau}  
\end{minipage}
   $\hspace{-2pt}\xrightarrow{\hspace{0.3in}
   }\quad$
   \begin{minipage}{10em}
        \begin{ytableau}
        m&m&a&m&m&m\\
        m&m&a&m&m\\
        m&m&a\\
        b&b\\
        m&m\\
        m\\
        m
        \end{ytableau}
        \end{minipage}
    \caption{Conjugation of an $(a,b,m)$-copartition.}
    \label{fig:my_label}
\end{figure}

\begin{remark}
Conjugation is a size-preserving bijection from $\CP_{a,b,m}$ to $\CP_{b,a,m}$.
\end{remark}

\begin{remark}
When $a\ne b$, there are no self-conjugate $(a,b,m)$-copartitions.
\end{remark}
\begin{theorem}\label{thm:selfConj}
Let $\mathrm{scp}_{a,m}(n)$ denote the number of self-conjugate $(a,a,m)$-copartitions of size $n$. Then, we have the following generating function:
\begin{equation*}
    \sum_{n=0}^\infty \mathrm{scp}_{a,m}(n)=(-q^{m+2a};q^{2m})_\infty.
\end{equation*}
\end{theorem}
\begin{proof}
We prove this theorem combinatorially by adapting Sylvester's proof that self-conjugate partitions of $n$ are equinumerous with partitions of $n$ into distinct odd parts \cite[p.~275]{SylvesterACT}.   Consider the graphical representation of a self-conjugate $(a,a,m)$-copartition. For each cell on the line $y=-x$, there is a corresponding hook that consists of that cell, all the cells directly below it, and all the cells directly to the right of it. Note that, because the diagram is self conjugate, each hook has an odd number of $m$'s and exactly two $a$'s. Thus, if we consider the hooks as the parts of a new partition, this new partition will have distinct parts that are congruent to $m+2a\pmod{2m}$. Similarly, given a partition into distinct parts congruent to $m+2a\pmod{2m}$, we can create a self-conjugate $(a,a,m)$ copartition. 
Since $(-q^{m+2a};q^{2m})_\infty$ generates partitions into distinct parts that are congruent to $m+2a\pmod{2m}$, our result follows.
\end{proof}

Since conjugation is an involution on $(a,a,b)$-copartitions, $\cp_{a,a,m}(n)$ is odd exactly when there are an odd number of self-conjugate $(a,a,m)$-copartitions of size $n$. Thus, by Theorem \ref{thm:selfConj},
\begin{equation}
\sum_{n=0}^\infty \cp_{a,a,m}(n)q^n\equiv (-q^{m+2a};q^{2m})_{\infty}\pmod{2}. \label{eq:parityGF}
\end{equation} This congruence  immediately implies the following facts about the parity of $\cp_{a,a,m}(n)$. 

\begin{cor}
For even $m$, $\cp_{a,a,m}(2n+1)\equiv 0\pmod{2}$. 
\end{cor}

\begin{cor}\label{cor:densityLowerBound1}
For even $m$,
\begin{equation*}
    \liminf_{n\to \infty}\frac{\#\{1\le k\le n\mid\cp_{a,a,m}(k) \text{ is even}\}}{n}\ge\frac{1}{2}.
\end{equation*}
\end{cor}

\begin{cor}\label{cor:densityLowerBound2}
For $m\equiv2\pmod{4}$ and odd $a$, 
$$\liminf_{n\to \infty}\frac{\#\{1\le k\le n\mid\cp_{a,a,m}(k) \text{ is even}\}}{n}\ge\frac{3}{4}.$$
\end{cor}
\begin{proof}
Note that, because $m\equiv 2\pmod{4}$ and $a$ is odd, $m+2a\equiv 0\pmod 4$ and $2m\equiv0\pmod{4}$. Thus, all self-conjugate copartitions must be of size $0\pmod 4$. 
\end{proof}

\begin{cor}\label{m=2a}
For $a$ odd,
$$ \lim_{n\to \infty}\frac{\#\{1\le k\le n\mid\cp_{a,a,2a}(k) \text{ is even}\}}{n}=1.$$
Moreover, $\cp_{a,a,2a}(k)$ is odd if and only if
$k$ is $4a$ times a pentagonal number; that is, if and only if  $k=2an(3n-1)$ for some integer $n$. 
\end{cor}
\begin{proof}
From \eqref{eq:parityGF}, we see that 
\begin{align*}
\sum_{n=0}^\infty \cp_{a,a,2a}(n)q^n&\equiv (-q^{4a};q^{4a})_{\infty}\pmod{2}\\
&\equiv (q^{4a};q^{4a})_\infty\\
&=\sum_{n=-\infty}^\infty (-1)^nq^{2an(3n-1)}.
\end{align*}
The final equality follows from Euler's pentagonal number theorem:
$$(q;q)_\infty = \sum_{n=-\infty}^\infty (-1)^nq^{n(3n-1)/2}.$$ Then, because the final series is lacunary, we have proved that $\cp_{a,a,2a}(n)$ is even with density $1$. 
\end{proof}

Corollaries \ref{cor:densityLowerBound1} and \ref{cor:densityLowerBound2} come from explicit sets upon which $\cp_{a,a,m}(n)$ is always even. 
Outside of those sets, empirical evidence suggests that the parity is equally balanced, which leads us to the following conjecture.

\begin{conj}\label{conj:evenmodda}
For even $m$ and odd $a$, 

 \begin{numcases}{ \lim_{n\to \infty}\frac{\#\{1\le k\le n\mid\cp_{a,a,m}(k) \text{ is even}\}}{n}=  }
  1 &\text{if } $m=2a$ \label{proved}\\
  \frac{3}{4} & if  $m\equiv 0\pmod{4}$ \label{m0mod4}\\
  \frac{7}{8} & otherwise. \label{OddAEvenmOthers}
  \end{numcases}

\end{conj}
\noindent
Note that \eqref{proved} is Corollary \ref{m=2a}.

On the other hand, for $m$ odd, as is the case with $p(n)$, we do not know explicit sets upon which $\cp_{a,a,m}(n)$ is always even, and we make the following conjecture. 
\begin{conj}\label{conj:oddm}
For odd $m$ and $\gcd(a,m)=1$, $\cp_{a,a,m}(n)$ is even (odd) with density $\frac{1}{2}$. 
That is, $$\lim_{n\to \infty}\frac{\#\{1\le k\le n\mid\cp_{a,a,m}(k) \text{ is even (odd)}\}}{n}=\frac{1}{2}.$$
\end{conj}

Table \ref{tab:conjevenmoda} provides some computational evidence for Conjecture \ref{conj:evenmodda}.
\renewcommand{\arraystretch}{1.5}

\begin{table}[h]
\begin{center}
    \begin{tabular}{>{$}c<{$}>{$}c<{$}>{$}c<{$}}\toprule
     n & 
     \frac{\#\{1\le k\le n\mid\cp_{3,3,4}(k) \text{ is even}\}}{n}  & 
     \frac{\#\{1\le k\le n\mid\cp_{1,1,6}(k) \text{ is even}\}}{n} \\ \midrule
     1000 & 0.765 & 0.871\\
     3000 & 0.752  & 0.875\\
     5000 & 0.753  & 0.874 \\
     7000 & 0.749  & 0.875\\
     9000 & 0.748  & 0.873\\
     11000 & 0.749 & 0.874\\
     13000 & 0.750 & 0.875\\
     15000 & 0.749 & 0.875
\end{tabular}
\end{center}\caption{The proportion of even values of $\cp_{a,a,m}(k)$ for $1\le k\le n$ when $(a,m)$ is $(3,4)$ and $(1,6)$.}\label{tab:conjevenmoda}
\end{table}

\section{The Parity of $\cp_{a,m-a,m}(n)$}
In this section, we explore the parity of the values of another special family of copartition functions: $\cp_{a,m-a,m}(n).$ Note that $$\mathbf{cp}_{a,m-a,m}(q)=\frac{(q^m;q^m)_\infty}{(q^a;q^m)_\infty(q^{m-a};q^m)_\infty}=\frac{(q^m;q^m)^2_\infty}{f(q^a,q^{m-a})},$$ where $f(x,y)=\sum_{n=-\infty}^\infty x^{n(n+1)/2}y^{n(n-1)/2}$ is Ramanujan's theta function. This form for $\mathbf{cp}_{a,m-a,m}(q)$ suggests that there is a wide range of analytic tools that we can use to study this family of functions. 

A very basic question one may ask is simply,
for which $a,m$ 
does the sequence $\{\cp_{a,m-a,m}(n)\}_{n=0}^\infty$ take both even and odd values infinitely often?

\begin{theorem}
For all $a,m$, 
the sequence $\{\cp_{a,m-a,m}(n)\}_{n=0}^\infty$ takes both even and odd values infinitely often.
\end{theorem}

\begin{proof}
Since $\cp_{a,m-a,m}(n) = \cp_{m-a,a,m}(n)$,
without loss of generality, we assume $0 < a \le m/2$.

If $a=m/2$, 
by \eqref{eq:cpgf} we have
\begin{align}\label{eq:cpgf2}
   {\mathbf{cp}}_{m/2,m/2,m}(q)&
   =\frac{(q^{m};q^m)_\infty}{(q^{m/2};q^m)_\infty(q^{m/2};q^m)_\infty}
   \equiv \frac{(q^{m};q^{m})_\infty}{(q^{m};q^{2m})_\infty} \equiv (q^{2m};q^{2m})_\infty \pmod 2.
\end{align}
By Euler's Pentagonal Number Theorem, we see that the right-hand side of \eqref{eq:cpgf2} takes both even and odd values infinitely often.

For $a \ne m/2$, 
by \eqref{eq:cpgf} we have
\begin{align}\label{eq:cpgf3}
   {\mathbf{cp}}_{a,m-a,m}(q)&
   =\frac{(q^{m};q^m)^2_\infty}{(q^{m-a};q^m)_\infty(q^a;q^m)_\infty(q^{m};q^m)_\infty}
   \equiv \frac{(q^{2m};q^{2m})_\infty}{(q^{m-a};q^m)_\infty(q^a;q^m)_\infty(q^{m};q^m)_\infty} \pmod 2.
\end{align}
By applying Jacobi's triple product identity to the denominator of \eqref{eq:cpgf3}, multiplying both sides by that denominator, and applying Euler's Pentagonal Number Theorem to the remaining right-hand side, we have 
\begin{align}\label{eq:cpgf4}
    {\mathbf{cp}}_{a,m-a,m}(q)
   \sum_{n=-\infty}^\infty (-1)^n q^{an + mn(n-1)/2}
   \equiv  \sum_{k=-\infty}^\infty q^{mk(3k-1)} \pmod 2.
\end{align}

Now suppose $\cp_{a,m-a,m}(n)$ has finitely many odd values, and let $d$ be the largest integer such that $\cp_{a,m-a,m}(d)$ is odd.
Notice that $\cp_{a,m-a,m}(0) = 1$ is also odd.
If $d=0$, then we have from \eqref{eq:cpgf4} that $\{an + mn(n-1)/2 \mid n \in \Z \} = \{mn(3n-1) \mid n \in \Z \}$, which is not possible.
For $d>0$, we derive a contradiction by showing that if we go out far enough, the left-hand side of \eqref{eq:cpgf4} has two close odd values, but the right-hand side does not.
A short computation shows that since $a \ne m/2$,
$an_1 + mn_1(n_1-1)/2 = an_2 + mn_2(n_2-1)/2$ if and only if $n_1=n_2$.
Let $N_0$ be so large that for $|n| \geq N_0$, consecutive values of each set
$\{an + mn(n-1)/2 \mid n \in \Z \}$ and $\{mn(3n-1) \mid n \in \Z \}$ are more than $d$ apart.
Let $N_1 > N_0$ be so large that 
$aN_1 + mN_1(N_1-1)/2 > mN_0(3N_0-1)$.
Then, since $\cp_{a,m-a,m}(0)$ and $\cp_{a,m-a,m}(d)$ are both odd and $N_1 > N_0$, the coefficients of $q^{aN_1 + mN_1(N_1-1)/2}$ and $q^{aN_1 + mN_1(N_1-1)/2 +d}$ on the left-hand side of  \eqref{eq:cpgf4} are both odd.
However, for exponents in that range, the terms of the right-hand side of \eqref{eq:cpgf4} with odd coefficients are more than $d$ terms apart, a contradiction.
Thus $\cp_{a,m-a,m}(n)$ has infinitely many odd values.

To show $\cp_{a,m-a,m}(n)$ is even infinitely often, define $E_{a,m}$ to be the set of nonnegative integers $n$ such that $\cp_{a,m-a,m}(n)$ is even, and define $G_{a,m}(q) = \sum_{n \in E_{a,m}} q^n$.
Notice
$$ \frac{1}{1-q} \equiv  {\mathbf{cp}}_{a,m-a,m}(q) + G_{a,m}(q) \pmod 2.$$
Multiplying both sides by the denominator in \eqref{eq:cpgf3} after applying Jacobi's triple product identity, we have 
\begin{align}\label{eq:evencpvalues} \frac{\sum_{n=-\infty}^\infty (-1)^n q^{an + mn(n-1)/2} }{1-q} &\equiv (q^{m};q^m)^2_\infty + G_{a,m}(q) \sum_{n=-\infty}^\infty (-1)^n q^{an + mn(n-1)/2} \pmod{2} \\
&\equiv (q^{2m};q^{2m})_\infty + G_{a,m}(q) \sum_{n=-\infty}^\infty q^{an + mn(n-1)/2}  \pmod 2.
\end{align}
We can rewrite the left-hand side of \eqref{eq:evencpvalues} by pairing  summands as
\begin{align}
   \frac{1}{1-q} &\sum_{n=0}^\infty (-1)^n [ q^{-an + mn(n+1)/2} - q^{a(n+1) + mn(n+1)/2} ] 
   = \sum_{n=0}^\infty (-1)^n q^{-an + mn(n+1)/2} \frac{1- q^{a(2n+1)}}{1-q} 
   \nonumber \\
   &\equiv \sum_{n=0}^\infty q^{-an + mn(n+1)/2} + q^{-an + mn(n+1)/2+1}+\cdots + q^{a(n+1) + mn(n+1)/2-1} \pmod 2, \label{eq:sumoffingeom}
\end{align}
where each term of the sum in \eqref{eq:sumoffingeom} has been expanded into a finite geometric series, the sum of $a(2n+1)$ consecutive powers of $q$.
These finite geometric series do not overlap, and so  
a short computation shows that \eqref{eq:sumoffingeom} has 
$aN^2$ odd terms up through the
$q^{\lfloor mN^2/2 \rfloor}$ term, so that the terms with odd coefficients in \eqref{eq:sumoffingeom} and the left-hand side of \eqref{eq:evencpvalues} have density $2a/m$.

By Euler's Pentagonal Number Theorem, the odd values of $(q^{2m};q^{2m})_\infty$ have density zero, thus the nonzero values of the product $G_{a,m}(q) \sum_{n=-\infty}^\infty (-1)^n q^{an + mn(n-1)/2} $  must have density $2a/m$. 
Since the nonzero values of $\sum_{n=-\infty}^\infty (-1)^n q^{an + mn(n-1)/2}$ have density zero, there must be infinitely many nonzero terms of $G_{a,m}(q)$.
\end{proof}

\begin{remark}
Some of the analysis above is very similar to that of Berndt, Yee, and Zaharescu \cite{BYZ} in treating the parity of $p(r,s;n)$, the number of partitions of $n$ into parts congruent to $r$, $s$, or $r+s$ modulo $r+s$.
In the second half of the proof above, tracking the number of odd terms as they did, one can arrive at the quantitative result that 
$\# \{ n < N \mid \cp_{a,m-a,m}(n) \text{ is even} \} > (a/\sqrt{2m}-o(1))\sqrt{N}$
for all $a \le m/2$.
In the special case $a=m/2$, we have $\# \{ n < N \mid \cp_{a,m-a,m}(n) \text{ is even} \} > (1-o(1))N$.
\end{remark}

Empirical evidence suggests that the parity of $\cp_{a,m-a,m}(n)$ may be balanced for many but not all $a,m$ with $\gcd(a,m)=1$.
This leads us to the following conjecture and question.

\begin{conj}\label{conj:oddm5050}
For odd $m$ and $\gcd(a,m)=1$, the sequence $\{\cp_{a,m-a,m}(n)\}_{n=0}^\infty$ takes both even and odd values with density $\frac{1}{2}$. 
That is, $$\lim_{n\to \infty}\frac{\#\{1\le k\le n\mid\cp_{a,m-a,m}(k) \text{ is even (odd)}\}}{n}=\frac{1}{2}.$$
\end{conj}

\begin{quest}\label{quest:abne}
For which $a,m$ with $\gcd(a,m)=1$ is $$\lim_{n\to \infty}\frac{\#\{1\le k\le n\mid\cp_{a,m-a,m}(k) \text{ is even}\}}{n} = \frac12?$$
Furthermore, when the limit is not $1/2$, what is 
$$\lim_{n\to \infty}\frac{\#\{1\le k\le n\mid\cp_{a,m-a,m}(k) \text{ is even}\}}{n} ?$$
When the limit is not $1/2$, is it always $1$?
\end{quest}

In Appendix \ref{sec:datatables}, we provide computational data surrounding Conjecture \ref{conj:oddm5050} and
Question \ref{quest:abne}.

\subsection{The parity of $\cp_{3,1,4}(n)$ and $\cp_{5,1,6}(n)$}

In two special cases, we can answer Question \ref{quest:abne}
and give the exact asymptotic density of $n$ for which $\cp_{a,b,m}(n)$ is even.
The generating functions for $\cp_{3,1,4}(n)$ and $\cp_{5,1,6}(n)$ have nice properties that allow us to apply the classical theory of binary quadratic forms to show that $\cp_{3,1,4}(n)$ and $\cp_{5,1,6}(n)$ are even on sets with arithmetic density one.

We now provide an explicit set with arithmetic density one on which $\cp_{3,1,4}(n)$ is even. 
\begin{theorem}\label{thm:cp314even}
When $a=3, b=1,$ and $m=4$,
$$\lim_{n\to \infty}\frac{\#\{1\le k\le n\mid\cp_{3,1,4}(k) \text{ is even}\}}{n} = 1.$$
In particular, $\cp_{3,1,4}(n)$ is even if  the prime factorization of $24n+5$ has a prime $\equiv 3 \pmod 4$ occurring with an odd exponent.
\end{theorem}

\begin{remark}
Note that $24n+5$ having a prime $\equiv 3 \pmod 4$ occurring with an odd exponent is a sufficient condition, but is not necessary.  
\end{remark}

\begin{proof}
We can rewrite the generating function for $\cp_{3,1,4}$ in a very useful form.
Consider 
 \begin{align}
 \sum_{n=0}^\infty \cp_{3,1,4}(n)q^n&=\frac{(q^{4};q^4)_\infty}{(q^3;q^4)_\infty(q;q^4)_\infty} =  \frac{(q^{4};q^4)_\infty}{(q;q^2)_\infty} \nonumber \\
  &=  \frac{(q^{4};q^4)_\infty(q^2;q^2)_\infty}{(q;q)_\infty} \equiv (q;q)_\infty(q^4;q^4)_\infty \pmod 2. 
\end{align}
Applying Euler's Pentagonal Number Theorem, we then have that
$$
\sum_{n=0}^\infty \cp_{3,1,4}(n)q^n \equiv 
 \left ( \sum_{j=-\infty}^\infty q^{j(3j+1)/2}  \right ) 
    \left ( \sum_{k=-\infty}^\infty q^{2k(3k+1)} \right )
\pmod 2.$$
Thus, $\cp_{3,1,4}(n)$ must be even unless 
$n = j(3j+1)/2 + 2k(3k+1)$ for some integers $j$ and $k$, or
equivalently, unless
$24n + 5 = (6j+1)^2 + 4(6k+1)^2$.
Analyzing this modulo $24$, this occurs if and only if $24n+5$ is represented by the form $A^2 + B^2$.
From the classical theory of binary quadratic forms, we know that the integers $24n+5$ representable by the form $A^2 + B^2$ are precisely those with prime factorizations having all powers of primes $\equiv 3 \pmod 4$ occurring with an even exponent.
Since that set of representable $(24n+5)$s has density zero, we have
$$\lim_{n\to \infty}\frac{\#\{1\le k\le n\mid\cp_{3,1,4}(k) \text{ is even}\}}{n} = 1$$ as desired.
\end{proof}

Considering the set of even values of $\cp_{3,1,4}$ guaranteed by Theorem \ref{thm:cp314even}, it is straightforward to write down arithmetic progressions on which $\cp_{3,1,4}$
is always even.

\begin{cor}\label{cor:cp3141}
For any prime $p>3, p\equiv 3 \pmod 4$, let $24 \delta \equiv 1 \pmod {p^2}$.
Then
$$
\cp_{3,1,4}(p^2k+pt-5\delta) \equiv 0 \pmod 2
$$
for $t=1, 2, \dots, p-1$ and every nonnegative integer $k$.
\end{cor}

Below we give two specific examples.

\begin{cor}\label{cor:cp3142}
For $r=3, 17, 24, 31, 38, 45$, we have 
$$ \cp_{3,1,4}(49k+r) \equiv 0 \pmod 2
$$
for every nonnegative integer $k$.
\end{cor}

\begin{cor}\label{cor:cp3143}
For $r=3, 14, 36, 47, 58, 69, 80, 91, 102, 113$, we have 
$$ \cp_{3,1,4}(121k+r) \equiv 0 \pmod 2
$$
for every nonnegative integer $k$.
\end{cor}

In order to treat $\cp_{5,1,6}(n)$, we use the same techniques, but the argument is slightly more complicated.
We first require a lemma showing the relationship between two binary quadratic forms.
\begin{lemma}\label{lem:formequivalence}
An integer $N \equiv 1 \pmod 6$ is representable by the binary quadratic form $A^2 + 3B^2$ if and only if $4N$ is representable by the binary quadratic form $(6J+1)^2 + 3(6K+1)^2$.
\end{lemma}

 \begin{proof}
 Suppose $N \equiv 1 \pmod 6$ is representable by the form $A^2 + 3B^2$.
 Then there exist $a,b \in \Z$ such that 
 \begin{equation}\label{eq:N=ab}
 N = a^2 + 3b^2.
 \end{equation}
 Specifically, because $a^2=(-a)^2$ and $a\not \equiv0\pmod{3}$, if $a \equiv b \pmod 3$, let us instead choose $a$ to be $-a$ in \eqref{eq:N=ab} to obtain a representation where $a\not\equiv b \pmod{3}$.
 Now notice 
 \begin{align}
\label{eq:4N}
 (a+3b)^2 + 3(a-b)^2 = 4a^2 + 12b^2 &= 4N \nonumber \\
 (a+3b)^2 + 3(b-a)^2 = 4a^2 + 12b^2 &= 4N \nonumber \\
 (-a-3b)^2 + 3(a-b)^2 = 4a^2 + 12b^2 &= 4N, \text{\ and} \nonumber \\
 (-a-3b)^2 + 3(b-a)^2 = 4a^2 + 12b^2 &= 4N.
 \end{align}
 Since $a$ and $b$ must be of opposite parity by (\ref{eq:N=ab}), we have one representation of $4N$ in (\ref{eq:4N}) of the form $(6J+1)^2 + 3(6K+1)^2$.

Now instead suppose $N$ is any positive integer such that $4N = (6j+1)^2 + 3(6k+1)^2$.
If $j \equiv k \pmod 2$,
let $a = (3j+9k)/2 + 1$ and $b = (3j-3k)/2$.
Then 
\begin{equation}
    a^2 + 3b^2 = \frac{36j^2 + 12j + 1}{4} + 3\frac{36k^2 + 12k + 1}{4} 
       = \frac{(6j+1)^2}{4} + 3\frac{(6k+1)^2}{4} = N.
\end{equation}
Otherwise, if $j \not \equiv k \pmod 2$,
let $a = (3j-9k-1)/2$ and $b = (3j+3k+1)/2$.
Then 
\begin{equation}
    a^2 + 3b^2 = \frac{36j^2 + 12j + 1}{4} + 3\frac{36k^2 + 12k + 1}{4} 
       = \frac{(6j+1)^2}{4} + 3\frac{(6k+1)^2}{4} = N.
\end{equation}
In either case, we see that $N$ is of the form $A^2 + 3B^2$, and the lemma follows.

\end{proof}

We now provide an explicit set with arithmetic density one on which $\cp_{5,1,6}(n)$ is even.

\begin{theorem}\label{thm:cp516even}
When $a=5, b=1,$ and $m=6$,
$$\lim_{n\to \infty}\frac{\#\{1\le k\le n\mid\cp_{5,1,6}(k) \text{ is even}\}}{n} = 1.$$
In particular, $\cp_{5,1,6}(n)$ is even if  the prime factorization of $6n+1$ has a prime $\equiv 2 \pmod 3$ occurring with an odd exponent.
\end{theorem}

\begin{remark}
Note that $6n+1$ having a prime $\equiv 2 \pmod 3$ occurring with an odd exponent is a sufficient condition, but is not necessary. For example, $\cp_{5,1,6}(5)=2$ is even, even though $6(5)+1=31$.
\end{remark}

\begin{proof}
We can rewrite the generating function for $\cp_{5,1,6}$ in a very useful form.
Consider 
 \begin{align}
 \frac{(q^2;q^2)_\infty(q^3;q^3)_\infty}{(q;q)_\infty(q^6;q^6)_\infty} 
&=\frac{(q^2;q^2)_\infty(q^3;q^6)_\infty(q^6;q^6)_\infty}{(q;q^2)_\infty(q^2;q^2)_\infty(q^6;q^6)_\infty} \nonumber \\
&=\frac{(q^3;q^6)_\infty}{(q;q^2)_\infty} \nonumber \\
&=\frac{1}{(q;q^6)_\infty(q^5;q^6)_\infty}. \nonumber \\
\end{align}
Thus, 
 \begin{align}
 \sum_{n=0}^\infty \cp_{5,1,6}(n)q^n&=\frac{(q^{6};q^6)_\infty}{(q^5;q^6)_\infty(q;q^6)_\infty} =  \frac{(q^2;q^2)_\infty(q^3;q^3)_\infty}{(q;q)_\infty} \nonumber \\
  &=  (-q;q)_\infty(q^3;q^3)_\infty \equiv (q;q)_\infty(q^3;q^3)_\infty \pmod 2. 
\end{align}
Applying Euler's Pentagonal Number Theorem, we then have that
$$
\sum_{n=0}^\infty \cp_{5,1,6}(n)q^n \equiv 
 \left ( \sum_{j=-\infty}^\infty q^{j(3j+1)/2}  \right ) 
    \left ( \sum_{k=-\infty}^\infty q^{3k(3k+1)/2} \right )
\pmod 2.$$
Thus, $\cp_{5,1,6}(n)$ must be even unless 
$n = j(3j+1)/2 + 3k(3k+1)/2$ for some integers $j$ and $k$, or
equivalently, unless
$(6j+1)^2 + 3(6k+1)^2 = 24n + 4$.
By Lemma \ref{lem:formequivalence}, this occurs if and only if $6n+1$ is represented by the form $A^2 + 3B^2$.
From the classical theory of binary quadratic forms \cite[Chapter 1]{primes_of_form}, we know that the integers $6n+1$ representable by the form $A^2 + 3B^2$ are precisely those with prime factorizations having all powers of primes $\equiv 2 \pmod 3$ occurring with an even exponent.
Since that set of representable $(6n+1)$s has density zero, we have
$$\lim_{n\to \infty}\frac{\#\{1\le k\le n\mid\cp_{5,1,6}(k) \text{ is even}\}}{n} = 1$$ as desired.
\end{proof}

Considering the set of even values of $\cp_{5,1,6}$ guaranteed by Theorem \ref{thm:cp516even}, it is straightforward to write down arithmetic progressions on which $\cp_{5,1,6}$
is always even.

\begin{cor}\label{cor:cp5161}
For any prime $p>2, p\equiv 2 \pmod 3$, let $6 \delta \equiv 1 \pmod {p^2}$.
Then
$$
\cp_{5,1,6}(p^2k+pt-\delta) \equiv 0 \pmod 2
$$
for $t=1, 2, \dots, p-1$ and every nonnegative integer $k$.
\end{cor}

Below we give two specific examples.

\begin{cor}\label{cor:cp5162}
For $r=9, 14, 19, 24$, we have 
$$ \cp_{5,1,6}(25k+r) \equiv 0 \pmod 2
$$
for every nonnegative integer $k$.
\end{cor}

\begin{cor}\label{cor:cp5163}
For $r=9, 31, 42, 53, 64, 75, 86, 97, 108, 119$, we have 
$$ \cp_{5,1,6}(121k+r) \equiv 0 \pmod 2
$$
for every nonnegative integer $k$.
\end{cor}

\section{Conclusion}
Conjectures \ref{conj:evenmodda}, \ref{conj:oddm}, and \ref{conj:oddm5050} and Question \ref{quest:abne} remain open.
Although we treat the parity of $\cp_{a,a,m}(n)$ and $\cp_{a,m-a,m}(n)$
in some detail above, the parity of $\cp_{a,b,m}(n)$ when $a \ne b$, and $a+b\ne m$ is largely unexplored.
With only a small amount of evidence, we make the following somewhat bold conjecture.
\begin{conj} 
When $\gcd(a,b,m)=1$, $a \ne b$, and $a+b\ne m$, $\cp_{a,b,m}(n)$ is even (odd) with density $\frac{1}{2}$.
That is,
$$\lim_{n\to \infty}\frac{\#\{1\le k\le n\mid\cp_{a,b,m}(k) \text{ is even}\}}{n}=\frac{1}{2}.$$
\end{conj}

In another direction, we give several congruences modulo 2 in Corollaries \ref{cor:cp3141}-\ref{cor:cp3143} and
\ref{cor:cp5161}-\ref{cor:cp5163},
In \cite{Andrews18}, Andrews gave a congruence modulo 5, namely 
$\mathcal{EO}^*(10n+8)\equiv 0\pmod{5}$, which, when written in the language of copartitions, becomes $\cp_{1,1,2}(5n+4)\equiv 0\pmod{5}$.
 Additionally, he defined an even-odd partition crank that witnesses this congruence. The equivalent copartition crank is the number of ground parts minus the number of sky parts, and the copartition crank witnesses the congruence $\cp_{1,1,2}(5n+4)\equiv 0\pmod{5}$. A combinatorial proof of these congruences would be a welcome addition to the literature.

\bibliographystyle{amsplain}
\bibliography{copartitions}

\appendix

\section{Data tables}\label{sec:datatables}

For readers interested in thinking about Conjecture \ref{conj:oddm5050} and  Question \ref{quest:abne}, we tabulate the proportion of even values of $\cp_{a,m-a,m}(k)$ for $k$ up to $32000$ for several values of $a$ and $m$.
In general, we find that the growth/behavior of the proportion seems roughly consistent when fixing $m$ and varying $a$.  For example, in Table \ref{tab:varya}, we see that the behavior in each column is roughly the same.
\begin{table}[h]
\begin{center}
    \begin{tabular}{>{$}c<{$}>{$}c<{$}>{$}c<{$}>{$}c<{$}}\toprule
     \quad n &    \frac{\#\{1\le k\le n\mid\cp_{1,11,14}(k) \text{ is even}\}}{n} & \frac{\#\{1\le k\le n\mid\cp_{3,11,14}(k) \text{ is even}\}}{n} & \frac{\#\{1\le k\le n\mid\cp_{5,9,14}(k) \text{ is even}\}}{n} \\
\midrule
   \quad   1000 & 0.535 & 0.543 & 0.530 \\
   \quad   2000 & 0.536 & 0.545 & 0.536 \\
   \quad   4000 & 0.549 & 0.552 & 0.543  \\
   \quad   8000 & 0.557 & 0.553 & 0.554 \\
   \quad   16000 & 0.568 & 0.564 & 0.565 \\
   \quad   32000 & 0.576 & 0.572 & 0.573 \\
\end{tabular}\caption{The proportion of even values of $\cp_{a,14-a,14}(k)$ for $1\le k\le n$ when $a$ is $1$, $3$, and $5$.}\label{tab:varya}
\end{center}
\end{table}
Since this is the case, in Table \ref{tab:varym} below, we just give $\displaystyle \frac{\#\{1\le k\le n\mid\cp_{1,m-1,m}(k) \text{ is even}\}}{n}$ for several $m$ between $3$ and $32$.

\begin{table}
\begin{center}
    \begin{tabular}{>{$}c<{$}>{$}c<{$}>{$}c<{$}>{$}c<{$}>{$}c<{$}>{$}c<{$}>{$}c<{$}>{$}c<{$}>{$}c<{$}>{$}c<{$}}\toprule
    n \backslash m & 3 & 4 & 5 & 6 & 7 & 8 & 9 & 10 & 12 \\
    \midrule
    1000 & 0.504 & 0.602 & 0.503 & 0.581 & 0.500 & 0.632 & 0.519 & 0.553 & 0.498\\
    2000 & 0.495 & 0.630 & 0.511 & 0.599 & 0.505 & 0.657 & 0.500 & 0.577 & 0.495 \\
    4000 & 0.495 & 0.656 & 0.509 & 0.623 & 0.500 & 0.681 & 0.505 & 0.593 & 0.499 \\
    8000 & 0.506 & 0.681 & 0.509 & 0.641 & 0.493 & 0.700 & 0.497 & 0.608 & 0.494 \\
    16000 & 0.503 & 0.701 & 0.508 & 0.653 & 0.496 & 0.719 & 0.497 & 0.625 & 0.499 \\
    32000 & 0.507 & 0.720 & 0.501 & 0.671 & 0.496 & 0.736 & 0.502 & 0.638 & 0.498 \\
    \midrule
    \end{tabular}
         \begin{tabular}{>{$}c<{$}>{$}c<{$}>{$}c<{$}>{$}c<{$}>{$}c<{$}>{$}c<{$}>{$}c<{$}>{$}c<{$}>{$}c<{$}>{$}c<{$}>{$}c<{$}}\toprule
    n \backslash m  & 14 & 16 & 18 & 20 & 22 & 24 & 26 & 28 & 30 & 32 \\
    \midrule
    1000  & 0.535 & 0.480 & 0.543 & 0.523 & 0.540 & 0.489 & 0.465 & 0.490 & 0.484 & 0.488\\
    2000 & 0.536 & 0.488 & 0.550 & 0.502 & 0.544 & 0.508 & 0.507 & 0.498 & 0.495 & 0.501 \\
    4000 & 0.549 & 0.497 & 0.547 & 0.499 & 0.550 & 0.503 & 0.517 & 0.502 & 0.503 & 0.503 \\
    8000 & 0.557 & 0.501 & 0.551 & 0.494 & 0.566 & 0.496 & 0.513 & 0.502 & 0.507 & 0.504 \\
    16000 & 0.568 & 0.504 & 0.554 & 0.496 & 0.566 & 0.499 & 0.509 & 0.501 & 0.509 & 0.502 \\
    32000 & 0.576 & 0.504 & 0.556 & 0.498 & 0.575 & 0.501 & 0.505 & 0.500 & 0.506 & 0.500 \\
         \midrule
    \end{tabular}
\end{center}
    \caption{The proportion of even values of $\cp_{1,m-1,m}(k)$ for $1\le k\le n$.}\label{tab:varym}
   \end{table}


\end{document}